\documentclass[12pt]{amsart}
\usepackage{amssymb,latexsym,eufrak,amsmath,amscd, graphicx}
\usepackage[colorlinks=true,pagebackref,hyperindex]{hyperref}

\usepackage{amsfonts}
\usepackage{amscd}

\usepackage{xypic}

\renewcommand{\phi}{\varphi}
\renewcommand{\epsilon}{\varepsilon}
\renewcommand{\theta}{\vartheta}

\def\ZZ{{\mathbf Z}}

\def\AAA{{\mathbf A}}
\def\RR{{\mathbf R}}

\def\PP{{\mathbf P}}

\def\cJ{\mathcal{J}}

\def\cC{\mathcal{C}}

\def\cO{\mathcal{O}}

\def\cM{\mathcal{M}}

\def\fra{\mathfrak{a}}
\def\frb{\mathfrak{b}}

\def\frm{\mathfrak{m}}

 \DeclareMathOperator{\Spec}{Spec}

\newtheorem{lemma}{Lemma}[section]
\newtheorem{theorem}[lemma]{Theorem}

\newtheorem{proposition}[lemma]{Proposition}

\newtheorem{conjecture}[lemma]{Conjecture}

\theoremstyle{definition}

\newtheorem{remark}[lemma]{Remark}

\theoremstyle{remark}
\newtheorem*{remark*}{Remark}
\newtheorem*{note*}{Note}

\begin{document}

\title{Ordinary varieties and the comparison between multiplier ideals
and test ideals II}

\thanks{2000\,\emph{Mathematics Subject Classification}.
 Primary 13A35; Secondary 14F18, 14F30.
\newline The author was partially supported by
 NSF grant DMS-0758454 and
  a Packard Fellowship.}

\author[M.~Musta\c{t}\u{a}]{Mircea~Musta\c{t}\u{a}}
\address{Department of Mathematics, University of Michigan,
Ann Arbor, MI 48109, USA}
\email{{mmustata@umich.edu}}

\begin{abstract}
We consider the following conjecture: if $X$ is a smooth $n$-dimensional projective variety 
in characteristic zero, then there is a dense set  of reductions $X_s$ to positive characteristic such that the action of the Frobenius morphism on $H^n(X_s,\cO_{X_s})$
is bijective. We also consider the conjecture relating the multiplier ideals of an ideal $\fra$ on a nonsingular variety in characteristic zero, and the test ideals of the reductions of $\fra$ to positive characteristic. We prove that
the latter conjecture implies the former one. 
\end{abstract}

\maketitle

\markboth{M.~MUSTA\c{T}\u{A}}{ORDINARY VARIETIES, MULTIPLIER IDEALS AND TEST IDEALS
II}

\section{Introduction}

This note is motivated by the joint paper with V.~Srinivas \cite{MS}, aimed at understanding 
the following conjecture relating invariants of singularities in characteristic zero with corresponding
invariants in positive characteristic. For a discussion of the notions involved, see below.

\begin{conjecture}\label{conj1}
Let $Y$ be a smooth, irreducible variety over an algebraically closed field $k$ of characteristic 
zero, and $\fra$ a nonzero ideal on $Y$. Given any model $Y_A$ and $\fra_A$ for $Y$ and $\fra$
over a subring $A$ of $k$, finitely generated over $\ZZ$, there is a dense set of closed points
$S\subset \Spec\,A$ such that
\begin{equation}\label{eq1_introd}
\cJ(Y,\fra^{\lambda})_s=\tau(Y_s,\fra_s^{\lambda})
\end{equation}
for every $\lambda\in\RR_{\geq 0}$ and every $s\in S$.
\end{conjecture}

In the conjecture, we denote by $Y_s$ the fiber of $Y_A$ over $s\in S$, and $\fra_s$ is the ideal
on $Y_s$ induced by $\fra_A$. The ideals $\cJ(Y,\fra^{\lambda})$ are the multiplier ideals of 
$\fra$. These are fundamental invariants of the singularities of $\fra$, that have seen a lot of recent applications due to their appearance in vanishing theorems (see \cite[Chapter 9]{positivity}).
The ideals $\tau(Y_s,\fra_s^{\lambda})$ are the (generalized) test ideals
of Hara and Yoshida \cite{HY}, defined in positive characteristic using  the
Frobenius morphism. The above conjecture asserts therefore that for a dense set of closed points,
we have the equality between the test ideals of $\fra$ and the reductions of the multiplier ideals of
$\fra$ for \emph{all} exponents. We note that it is shown in \cite{HY} that if $\lambda\in\RR_{\geq 0}$
is fixed, then the equality in (\ref{eq1_introd}) holds for every $s$ in an open subset of the closed points in $\Spec\,A$.

The following conjecture was proposed in \cite{MS}. 

\begin{conjecture}\label{conj2}
Let $X$ be a smooth, irreducible  $n$-dimensional projective variety defined over an algebraically closed field 
$k$ of characteristic zero. If $X_A$ is a model of $X$ defined over a subring $A$ of $k$, finitely generated over $\ZZ$,
then there is a dense set of closed points $S\subseteq \Spec\,A$ such that the Frobenius action on
$H^n(X_s,\cO_{X_s})$ is bijective for every $s\in S$. 
\end{conjecture}

It is expected, in fact, that there is a set $S$ as in Conjecture~\ref{conj2} such that 
$X_s$ is ordinary in the sense of Bloch and Kato \cite{BK} for every $s\in S$. In particular, this
would imply that the action of the Frobenius on each cohomology group $H^i(X_s,\cO_{X_s})$
is bijective (see \cite[Remark~5.1]{MS}). The main result of \cite{MS} is that Conjecture~\ref{conj2}
implies Conjecture~\ref{conj1}. In this note we show that the converse is true:

\begin{theorem}\label{main}
If Conjecture~\ref{conj1} holds, then so does Conjecture~\ref{conj2}.
\end{theorem}

The following is an outline of the proof. Given a variety $X$ as in Conjecture~\ref{conj2}, 
we embed it in a projective space $\PP_k^N$ such that $r:=N-n\geq n+1$, and the ideal $\fra
\subseteq k[x_0,\ldots,x_N]$ defining $X$ is generated by quadrics. In this case it is easy to 
compute the multiplier ideals $\cJ(\AAA_k^{N+1},\fra^{\lambda})$ for $\lambda<r$, and in particular we see that
$(x_0,\ldots,x_N)^{2r-N-1}\subseteq\cJ(\AAA_k^{N+1},\fra^{\lambda})$ for every $\lambda<r$.
It follows from a general property of multiplier ideals that
if $g_1,\ldots,g_r$ are general linear combinations of a system of generators of $\fra$, 
and if $h=g_1\cdots g_r$, then $\cJ(\AAA_k^{N+1}, \fra^{\lambda})=
\cJ(\AAA_k^{N+1}, h^{\lambda/r})$ for every $\lambda<r$. In this case, Conjecture~\ref{conj1}
implies that for a dense set of closed points $s\in\Spec\,A$, the ideal
$(x_0,\ldots,x_N)^{2r-N-1}$ is contained in $\tau(\AAA^{N+1}_{k(s)}, h_s^{\mu})$ for every $\mu<1$. Using some basic properties 
of test ideals, we deduce that the  Frobenius action on 
$H^{N-1}(D_s,\cO_{D_s})$ is bijective, where $D_s\subset\PP_{k(s)}^N$ is the hypersurface defined by $h_s$. We show that this in turn implies the bijectivity of the Frobenius action on 
$H^n(X_s,\cO_{X_s})$, hence proves the theorem.

\section{Proof of the main result}

We start by recalling the definition of multiplier ideals and test ideals. Suppose first that $Y$ is a smooth, irreducible variety
over an algebraically closed field $k$ of characteristic zero, and $\fra$ is a nonzero ideal on $Y$.
A \emph{log resolution} of $\fra$ is a projective, birational morphism $\pi\colon W\to Y$, with
$W$ smooth, such that $\fra\cdot\cO_W$ is the ideal of a divisor $D$ on $W$, with
$D+K_{W/Y}$ having simple normal crossings (as usual, $K_{W/Y}$ denotes the relative canonical divisor of $W$ over $Y$).
With this notation, for every $\lambda\in\RR_{\geq 0}$ we have
\begin{equation}\label{def_multiplier}
\cJ(Y,\fra^{\lambda})=\pi_*\cO_W(K_{W/Y}-\lfloor\lambda D\rfloor).
\end{equation}
Recall that if $E=\sum_ia_i E_i$ is a divisor with $\RR$-coefficients, then
$\lfloor E\rfloor=\sum_i\lfloor a_i\rfloor E_i$, where $\lfloor t\rfloor$ is the largest integer
$\leq t$.  It is a well-known fact that the above definition is independent of the choice 
of log resolution. For this and other basic facts about multiplier ideals, see
\cite[Chapter~9]{positivity}.

Suppose now that $Y=\Spec\,R$ 
is an affine smooth, irreducible scheme of finite type over a perfect field $L$
of positive characteristic $p$ (in the case of interest for us, $L$ will be a finite field). Under these assumptions, the test ideals admit the following simple description that we will use,
see \cite{BMS}. Recall that for an ideal $J$ and for $e\geq 1$, one denotes by $J^{[p^e]}$
the ideal $(h^{p^e}\mid h\in J)$. One can show that given an ideal $\frb$ in $R$, there is a unique
smallest ideal $J$ such that $\frb\subseteq J^{[p^e]}$; this ideal is denoted by 
$\frb^{[1/p^e]}$.
Suppose now that $\fra$ is an ideal in $R$ and $\lambda\in\RR_{\geq 0}$. One can show that for every $e\geq 1$ we have the inclusion 
$$(\fra^{\lceil \lambda p^e\rceil})^{[1/p^e]}\subseteq
(\fra^{\lceil \lambda p^{e+1}\rceil})^{[1/p^{e+1}]},$$
where $\lceil t\rceil$ denotes the smallest integer $\geq t$. Since $R$ is Noetherian, it follows
that $(\fra^{\lceil \lambda p^e\rceil})^{[1/p^e]}$ is constant for $e\gg 0$. This is the test ideal
$\tau(Y, \fra^{\lambda})$. For details and a discussion of basic properties of test ideals in this setting, we refer to \cite{BMS}. For a comparison of  general properties of multiplier ideals and test ideals,
see \cite{HY} and \cite{MY}.

If $\fra$ is an ideal in the polynomial ring $k[x_0,\ldots,x_N]$, where $k$ is a field of characteristic zero, a \emph{model} of $\fra$ over a subring $A$ of $k$, finitely generated over $\ZZ$, is an ideal
$\fra_A$ in $A[x_0,\ldots, x_N]$ such that $\fra_A\cdot k[x_0,\ldots,x_N]=\fra$. We can obtain
such a model by simply taking $A$ to contain all the coefficients of a finite system of generators of 
$\fra$. Of course, we may always replace $A$ by a larger ring with the same properties; in particular,
we may replace $A$ by a localization $A_a$ at a nonzero element $a\in A$.
If $s\in\Spec\,A$ and if $\fra_A$ is a model of $\fra$, then we obtain a corresponding ideal
$\fra_s$ in $k(s)[x_0,\ldots,x_N]$. Note that if $s$ is a closed point, then the residue field  $k(s)$
is a finite field. 

Suppose now that $X\subseteq\PP_k^{N}$ is a projective subscheme defined by the homogeneous ideal $\fra\subseteq k[x_0,\ldots,x_N]$. If $\fra_A\subseteq A[x_0,\ldots,x_N]$ is a model of $\fra$
over $A$, which we may assume homogeneous, then the subscheme $X_A$
of $\PP_A^{N}$ defined by $\fra_A$ is a model of $X$ over $A$. If $s\in\Spec\,A$, then the subscheme 
$X_s\subseteq\PP_{k(s)}^{N}$ is defined by the ideal $\fra_s$. We refer to \cite[\S 2.2]{MS}
for some of the standard facts about reduction to positive characteristic. We note that given $\fra$ as above,
we may consider simultaneously all the reductions $\cJ(\AAA_k^{N+1},\fra^{\lambda})_s$
for all $\lambda\in\RR_{\geq 0}$. This is due to the fact that for bounded $\lambda$ we only 
have to deal with finitely many ideals, while for $\lambda\gg 0$, the multiplier ideals are determined by the lower ones via a Skoda-type theorem (see \cite[\S 3.2]{MS} for details). 

We can now give the proof of our main result stated in Introduction.

\begin{proof}[Proof of Theorem~\ref{main}]
Let $X$ be a smooth, irreducible $n$-dimensional projective variety over an algebraically closed field $k$ of characteristic zero, with $n\geq 1$. It is clear that the assertion we need is independent of the model $X_A$ that we choose.
Consider a closed embedding $X\hookrightarrow \PP_k^{N}$. After replacing this by a composition with a 
$d$-uple Veronese embedding, for $d\gg 0$, we may assume that the saturated ideal 
$\fra\subset R=k[x_0,\ldots,x_N]$ defining $X$ is generated by homogeneous polynomials of degree
two (see \cite[Proposition~5]{ERT}). Furthermore, we may clearly assume that $r:=N-n\geq n+1$. Under these assumptions, it is easy to determine the multiplier ideals of $\fra$ of exponent $<r$.

\begin{lemma}\label{multiplier_ideals}
With the above notation, if $\frm=(x_0,\ldots,x_{N})$, then 
\[
\cJ(\AAA_k^{N+1}, \fra^{\lambda}) = \left\{
\begin{array}{cl}
R, & \text{if}\,\,\,0 \le \lambda < \frac{N+1}{2}; \\[2mm]
\frm^{\lfloor 2\lambda\rfloor -N}, & \text{if} \,\,\, \frac{N+1}{2} \le \lambda < r.
\end{array}\right.
\]
\end{lemma}

\begin{proof}
Let us fix $\lambda\in\RR_{\geq 0}$, with $\lambda<r$.
We denote by $Z$ the subscheme of $\AAA_k^{N+1}$ defined by $\fra$.
Let $\phi\colon W\to\AAA_k^{N+1}$ be the blow-up of the origin, with exceptional divisor $E$. 
Since $\fra$ is generated by homogeneous polynomials of degree two, it follows that 
$\fra\cdot\cO_W=\cO_W(-2E)\cdot\widetilde{\fra}$, where $\widetilde{\fra}$ is the ideal defining the 
strict transform $\widetilde{Z}$ of $Z$ on $W$. We have $K_{W/\AAA_k^{N+1}}=NE$, hence the change of variable formula for multiplier ideals (see \cite[Theorem~9.2.33]{positivity}) implies
\begin{equation}\label{eq1_lem_multiplier_ideals}
\cJ(\AAA_k^{N+1},\fra^{\lambda})=\phi_*\left(\cJ(W,(\fra\cdot\cO_W)^{\lambda})\otimes
\cO_W(NE)\right).
\end{equation}

It is clear that $\widetilde{Z}$ is nonsingular over 
$\AAA_k^{N+1}\smallsetminus\{0\}$. Since $\widetilde{Z}\cap E\subseteq E\simeq\PP^{N}$ is isomorphic to the
scheme $X$, hence it is nonsingular, it follows that $\widetilde{Z}$ is nonsingular,
and $\widetilde{Z}$ and $E$ have simple normal crossings.
Let $\psi\colon \widetilde{W}\to W$ be the blow-up of $W$ along $\widetilde{Z}$, with exceptional divisor  $T$, and let $\widetilde{E}$ be the strict transform of $E$.
Note that
$\widetilde{W}$ is nonsingular, and $\widetilde{E}+T$ has simple normal crossings. We have 
$K_{\widetilde{W}/W}=(r-1)T$ and $\fra\cdot\cO_{\widetilde{W}}=\cO_{\widetilde{W}}(-2\widetilde{E}-T)$. Therefore $\psi$ is a log resolution of $\fra\cdot\cO_W$, and by definition we have
\begin{equation}\label{eq2_lem_multiplier_ideals}
\cJ(W,(\fra\cdot\cO_W)^{\lambda})=\psi_*(\cO_{\widetilde{W}}(-(\lfloor\lambda\rfloor-r+1)T-\lfloor
2\lambda\rfloor \widetilde{E})
=\cO_W(-\lfloor 2\lambda\rfloor E)
\end{equation}
(recall that $\lambda<r$).
The formula in the lemma follows from (\ref{eq1_lem_multiplier_ideals}), (\ref{eq2_lem_multiplier_ideals}), and the fact that $\phi_*(\cO_W(-iE))=\frm^i$ for every 
$i\in\ZZ_{\geq 0}$.
\end{proof}

Let $f_1,\ldots,f_m$ be a system of generators of $\fra$, with each $f_i$ homogeneous of degree
two. We fix $g_1,\ldots,g_r$ general linear combinations of the $f_i$ with coefficients in $k$, and put
$h=g_1\cdots g_r$. In this case, we have
\begin{equation}\label{eq1_main}
\cJ(\AAA_k^{N+1},\fra^{\lambda})=\cJ(\AAA_k^{N+1}, h^{\lambda/r})
\end{equation}
for every $\lambda<r$ (see \cite[Proposition 9.2.28]{positivity}).

Suppose now that $\fra_A$ and $h_A$ are homogeneous models of $\fra$, and respectively $h$, over $A$. Let 
$X_A, D_A\subset\PP_A^N$ be the projective schemes defined by $\fra_A$ and
$h_A$, respectively. 
Note that $g_1,\ldots,g_r$ being general linear combinations of the $f_i$, the subscheme
$V(g_1,\ldots,g_r)\subset\PP^N_k$ has pure codimension $r$. Therefore we may assume that
for every $s\in\Spec\,A$,  the scheme $V((g_1)_s,\ldots,(g_r)_s)$ has pure codimension $r$
in $\PP^N_{k(s)}$. 
We need to show that given models as above, there is a dense set of
closed points
$S\subset \Spec\,A$ such that the Frobenius action on
$H^n(X_s,\cO_{X_s})$ is bijective for every $s\in S$. The next lemma shows that in fact, it is enough to find
$S$ as above such that the Frobenius action on $H^{N-1}(D_s,\cO_{D_s})$ is bijective
for all $s\in S$.

\begin{lemma}\label{lem_Frobenius_action}
Let $L$ be a finite field, and $D_1,\ldots,D_r$ hypersurfaces in $\PP^N=\PP_L^{N}$, with
$r\leq N$, such that
the intersection scheme $Y=D_1\cap\ldots\cap D_r$ has pure codimension $r$ in
$\PP^{N}$.
If the Frobenius action on $H^{N-1}(D,\cO_D)$ is bijective, where $D=\sum_{i=1}^rD_i$, then
for every closed subscheme $X$ of $Y$,
the Frobenius action on $H^{N-r}(X,\cO_X)$ is bijective.
\end{lemma}

\begin{proof}
If $r=N$, then $X$ is zero-dimensional, and the Frobenius action on 
$\Gamma(X,\cO_X)$ is bijective since $L$ is perfect.
Therefore from now on we may assume that $r\leq N-1$. 

For every subset $J\subseteq\{1,\ldots,r\}$, let $D_J=\bigcap_{j\in J}D_j$. 
By assumption, $Y$ is a complete intersection, hence 
there is an exact complex
$$\cC^{\bullet}:\,\,\,\,0\to \cC^0\overset{d^0}\to \cC^1\overset{d^1}\to\ldots\overset{d^{r-1}}\to \cC^r\to 0,$$
where $\cC^0=\cO_D$, and $\cC^m=\bigoplus_{|J|=m}\cO_{D_J}$ for $m\geq 1$. Note that 
we have a 
morphism of complexes $\cC^{\bullet}\to F_*(\cC^{\bullet})$, where $F$ is the absolute Frobenius morphism on $X$.
It follows that if we break-up $\cC^{\bullet}$
into short exact sequences, the maps in the corresponding long exact sequences for cohomology
are compatible with the Frobenius action.

Let $\cM^i={\rm Im}(d^i)$, hence $\cM^0\simeq \cC^0=\cO_D$ and $\cM^{r-1}=\cC^r=\cO_Y$.
Since each $D_J$ is a complete intersection in $\PP^N$, it follows that 
$H^i(D_J,\cO_{D_J})=0$ for every $i$ with $1\leq i<\dim(D_J)=N-|J|$.
We deduce that for every $i$ with $0\leq i\leq r-2$,
the short exact sequence
$$0\to \cM^{i}\to \cC^{i+1}\to\cM^{i+1}\to 0$$
gives an exact sequence 
$$0=H^{N-i-2}(\PP^{N}, \cC^{i+1})\to H^{N-i-2}(\PP^N, \cM^{i+1})\to H^{N-i-1}(\PP^N, \cM^i).$$
Therefore we have a sequence of injective maps
$$H^{N-r}(Y, \cO_Y)\hookrightarrow H^{N-r+1}(\PP^N, \cM^{r-2})\hookrightarrow\ldots
\hookrightarrow  H^{N-2}(\PP^N,\cM^1)\hookrightarrow H^{N-1}(D,\cO_D),$$
compatible with the Frobenius action. Since this action is bijective on $H^{N-1}(D,\cO_D)$
by hypothesis, it follows that it is bijective also on $H^{N-r}(Y,\cO_Y)$ 
(see, for example, \cite[Lemma~2.4]{MS}).

On the other hand, since $\dim(Y)=N-r$, the surjection $\cO_Y\to\cO_X$ induces a surjection
$H^{N-r}(Y,\cO_Y)\to H^{N-r}(X,\cO_X)$, compatible with the Frobenius action. As we have seen,
the Frobenius action is bijective on $H^{N-r}(Y,\cO_Y)$, hence on every quotient
(see \cite[Lemma~2.4]{MS}). This completes the proof of the lemma.
\end{proof}

Returning to the proof of Theorem~\ref{main}, we see that it is enough to show that 
there is a dense set of closed points $S\subset\Spec\,A$ such that
Frobenius acts bijectively on $H^{N-1}(D_s,\cO_{D_s})$ for $s\in S$.
We assume that Conjecture~\ref{conj1} holds, hence there is a 
dense set of closed points
$S\subset \Spec\,A$ such that $\tau(\AAA^{N+1}_{k(s)}, h_s^{\lambda})=
\cJ(\AAA_k^{N+1},h^{\lambda})_s$ for every $\lambda\in\RR_{\geq 0}$ and every $s\in S$. 
In particular, it follows from Lemma~\ref{multiplier_ideals} and (\ref{eq1_main}) 
that $(x_0,\ldots,x_N)^{2r-N-1}\subseteq \tau(\AAA_{k(s)}^{N+1}, h_s^{\lambda})$ for 
every $\lambda<1$. 
Since $\deg(h_s)=2r\geq (N+1)$,
 Proposition~\ref{key_ingredient} below implies that the Frobenius action on 
$H^{N-1}(D_s,\cO_{D_s})$ is bijective for all $s\in S$. As we have seen, this completes 
the proof of Theorem~\ref{main}.
\end{proof}

\begin{proposition}\label{key_ingredient}
Let $L$ be a perfect field of characteristic $p>0$, and $h\in R=L[x_0,\ldots,x_N]$  a homogeneous polynomial of degree $d\geq N+1$, with $N\geq 2$. If $(x_0,\ldots,x_N)^{d-N-1}\subseteq\tau(\AAA_L^{N+1},
h^{1-\frac{1}{p}})$, then the Frobenius action on 
$H^{N-1}(D,\cO_D)$ is bijective, where $D\subset\PP_L^{N}$ is the hypersurface defined by $h$.
\end{proposition}

\begin{proof}
In the case $d=N+1$, this is a reformulation of a well-known fact due to Fedder \cite{Fedder}.
We follow the argument from \cite[Proposition~2.16]{MTW}, that extends to our more general setting. 
It is enough to show that the Frobenius action on $H^{N-1}(D,\cO_D)$ is injective
(see \cite[\S 2.1]{MS}).

Note first that $\tau(\AAA_L^{N+1},
h^{1-\frac{1}{p}})=(h^{p-1})^{[1/p]}$ (see \cite[Lemma~2.1]{BMS1}), hence by assumption $\frm^{d-N-1}\subseteq (h^{p-1})^{[1/p]}$, where $\frm=(x_0,\ldots,x_N)$.
It is convenient to use the interpretation of the ideal $(h^{p-1})^{[1/p]}$ in terms of local cohomology.  
Let $E=H_{\frm}^{N+1}(R)$. Recall that this is a graded $R$-module, carrying a natural action of the Frobenius, that we denote by $F_E$. There is an isomorphism 
$$E\simeq R_{x_0\cdots x_N}/\sum_{i=0}^N R_{x_0\cdots\widehat{x_i}\cdots x_N}.$$
Via this isomorphism, $F_E$ is induced by the Frobenius morphism on $R_{x_0\cdots x_N}$.

The annihilator of $(h^{p-1})^{[1/p]}$ in $E$ is equal 
to ${\rm Ker}(h^{p-1}F_E)$ (see, for example, \cite[\S 2.3]{BMS}). Therefore
we have
\begin{equation}\label{eq1_key}
{\rm Ker}(h^{p-1}F_E)\subseteq {\rm Ann}_E(\frm^{d-N-1})=\bigoplus_{i\geq -d+1}E_i.
\end{equation}

On the other hand, the exact sequence 
$$0\to R(-d)\overset{h}\to R\to R/(h)\to 0$$
induces an isomorphism
$$H_{\frm}^N(R/(h))\simeq \{u\in E\mid hu=0\}(-d),$$
such that the Frobenius action on $H_{\frm}^N(R/(h))$ is given by
$h^{p-1}F_E$. 
Since $H^{N-1}(D,\cO_D)\simeq H_{\frm}^N(R/(h))_0\hookrightarrow
E_{-d}$, (\ref{eq1_key}) implies that the Frobenius action is injective on 
$H^{N-1}(D,\cO_D)$. This completes the proof of the proposition. 
\end{proof}

\begin{remark}
In the proof of Theorem~\ref{main} we only used the inclusion ``$\subseteq$" in Conjecture~\ref{conj1}. However, this is the interesting inclusion: the reverse one is known, see \cite{HY} or \cite[Proposition~4.2]{MS}.
It is more interesting that we only used 
Conjecture~\ref{conj1} when $Y=\AAA_k^{N+1}$, $\fra$ is principal and homogeneous, 
and $\lambda=1-
\frac{1}{p}$.  By combining Theorem~\ref{main} with the main result in \cite{MS}, we see that in order to prove Conjecture~\ref{conj1} in general, it is enough to consider the case when 
$Y=\AAA_k^n$, $\fra=(f)$ is principal and homogeneous, and show the following: if $\frb=\cJ(Y,\fra^{1-\epsilon})$
for $0<\epsilon\ll 1$, and if $f_A\in A[x_1,\ldots,x_n]$ is a model for $f$, then there is a dense set
of closed points $S\subset\Spec\,A$ such that
$$\frb_s\subseteq (f_s^{p-1})^{[1/p]}$$
for every $s\in S$, where $p={\rm char}(k(s))$.
\end{remark}

\providecommand{\bysame}{\leavevmode \hbox \o3em
{\hrulefill}\thinspace}


\begin{thebibliography}{BMS}

\bibitem[BMS1]{BMS1}
M.~Blickle, M.~Musta\c{t}\u{a}, and K.~E. Smith, $F$-thresholds of hypersurfaces,
Trans. Amer. Math. Soc. \textbf{361} (2009), 6549--6565.

\bibitem[BMS2]{BMS}
M.~Blickle, M.~Musta\c{t}\u{a}, and K.~E. Smith, Discreteness and rationality of $F$-thresholds,
Michigan Math. J. \textbf{57} (2008), 463--483.


\bibitem[BK]{BK}
S.~Bloch and K.~Kato, $p$-adic \'{e}tale cohomology, Inst. Hautes \'{E}tudes Sci.
Publ. Math. \textbf{63} (1986), 107--152.

\bibitem[ERT]{ERT}
D.~Eisenbud, A.~Reeves, and B.~Totaro, Initial ideals, Veronese subrings, and rates of algebras, Adv. Math. \textbf{109} (1994), 168--187.

\bibitem[Fe]{Fedder}
R.~Fedder, $F$-purity and rational singularity, Trans. Amer. Math. Soc. \textbf{278}
(1983), 461--480.

\bibitem[HY]{HY}
N.~Hara and K.-i.~Yoshida,
 A generalization of tight closure and multiplier
 ideals, Trans. Amer. Math. Soc. \textbf{355} (2003),
 3143--3174.


\bibitem[Laz]{positivity}
R.~Lazarsfeld, \textit{Positivity in algebraic geometry} II, Ergebnise der Mathematik und ihrer
Grenzgebiete, 3. Folge, Vol. 49, Springer-Verlag, Berlin, 2004.


\bibitem[MS]{MS}
M.~Musta\c{t}\u{a} and V.~Srinivas, 
Ordinary varieties and the comparison between multiplier ideals
and test ideals, arXiv:1012.2818, to appear in Nagoya Math. J.


\bibitem[MTW]{MTW}
M.~Musta\c{t}\u{a}, S.~Takagi, and K.-i. Watanabe, $F$-thresholds and Bernstein-Sato polynomials,
\emph{European Congress of Mathematic}, 341--364, Eur. Math. Soc., Z\"{u}rich, 2005.



\bibitem[MY]{MY}
M.~Musta\c{t}\u{a} and K.-i. Yoshida, Test ideals vs. multiplier ideals, Nagoya Math. J. \textbf{193}
(2009), 111-128.


\end{thebibliography}
\end{document}